\documentclass[11pt,a4paper]{amsart}
\allowdisplaybreaks[1]
\usepackage{mathtools}
\mathtoolsset{showonlyrefs}
\usepackage{amssymb}
\usepackage{mathrsfs}
\usepackage{amsbsy,amsgen,amscd,amsthm}
\usepackage{bm}
\usepackage{dsfont}
\usepackage{graphicx}
\usepackage{subfig}
\usepackage{caption}
\usepackage{float}
\usepackage{enumerate}
\usepackage{enumitem}
\usepackage{setspace}
\usepackage{tikz}
\usepackage[top=3.2cm,bottom=3.8cm,left=3cm,right=2cm]{geometry}
\usepackage{xcolor}
\usepackage[T1]{fontenc}
\usepackage{lmodern}
\usepackage[unicode=true,colorlinks,linkcolor=blue,citecolor=blue,urlcolor=blue,pagebackref]{hyperref}

\theoremstyle{plain}
\newtheorem{theorem}{Theorem}[section]
\newtheorem{lemma}[theorem]{Lemma}

\newtheorem{problem}{Problem}

\theoremstyle{definition}
\newtheorem{definition}[theorem]{Definition}

\theoremstyle{remark}
\newtheorem{remark}[theorem]{Remark}

\numberwithin{equation}{section}

\newcommand{\R}{\mathbb{R}}
\newcommand{\Sph}{\mathbb{S}}
\newcommand{\Ball}{\mathbb{B}}
\newcommand{\cP}{\mathcal{P}}
\newcommand{\cE}{\mathcal{E}}
\newcommand{\cA}{\mathcal{A}}
\newcommand{\Aut}{\operatorname{Aut}}
\newcommand{\vol}{\operatorname{vol}}
\newcommand{\degree}{\operatorname{deg}}
\newcommand{\supp}{\operatorname{supp}}
\newcommand{\Hom}{\operatorname{Hom}}
\newcommand{\End}{\operatorname{End}}
\newcommand{\GL}{\operatorname{GL}}
\newcommand{\OO}{\operatorname{O}}
\newcommand{\Lip}{\operatorname{Lip}}
\newcommand{\Int}{\operatorname{int}}
\newcommand{\Span}{\operatorname{span}}
\newcommand{\Sym}{\operatorname{Sym}}

\newcommand{\dd}{\,d}
\newcommand{\ip}[2]{\langle #1,#2\rangle}

\begin{document}
	\enlargethispage{3pt}
	
	\begin{center}
		
		{\Large \bf A solution to Banach's isometric conjecture}
		\vskip 20pt
		\begin{center}
			{{\bf    Xinbao~~Lu \quad \quad Kaiwen~~Yang}\\~~ \\ \small{School of Mathematical Sciences, Key Laboratory of Intelligent Computing and Applications (Tongji University), Ministry of Education, Tongji University, Shanghai, 200092, China}}
		\end{center}
	\end{center}
	
	\vskip 5pt

	\footnotetext{Email: xinbaolu@tongji.edu.cn; yangkaiwen@tongji.edu.cn.}
	\footnotetext{Research of the first author was supported by NSFC No. 12501078.}
	
	\vskip 5pt
	\begin{center}
		\begin{minipage}{14cm}
			{{\bf Abstract.}
				Banach asked in 1932 whether a real Banach space $X$ whose $n$-dimensional subspaces, for some fixed $1<n<\dim X$, are all linearly isometric must be a Hilbert space. 
				Gromov proved the conjecture for even $n$, and subsequent work settled several odd-dimensional cases. 
				We prove the conjecture for every odd $n$, including all previously unresolved cases. Together with Gromov's result for even $n$, 
				this completes Banach's isometric conjecture in the real case. The proof combines principal bundle theory with a Brouwer degree argument.
			}
			
			
		\end{minipage}
	\end{center}
	
	\vskip 25pt
	\section{\bf Introduction}
	\label{sec:introduction}
	\vskip 5pt
	
	Banach posed the following problem in 1932
	\cite[Chapter~XII, Remarks, property~(5), p.~244]{Banach1932}.
	
	\begin{problem}[Banach]\label{prob:banach}
		Let $X$ be a real or complex Banach space. If, for some fixed integer $2\leq n<\dim X$, all $n$-dimensional linear subspaces of $X$ are linearly isometric, must $X$ be a Hilbert space?
	\end{problem}
	
	An affirmative answer to Problem~\ref{prob:banach} is commonly called Banach's isometric conjecture.
	
	It suffices to prove the conjecture in the finite-dimensional codimension-one case $\dim X=n+1$. Indeed, assume that the conjecture holds for $(n+1)$-dimensional spaces, and let $X$ satisfy the hypothesis. If $Y\subset X$ is $(n+1)$-dimensional, then all $n$-dimensional subspaces of $Y$ are $n$-dimensional subspaces of $X$ and hence are mutually linearly isometric. The assumed case therefore implies that $Y$ is a Hilbert space. Since every pair of vectors in $X$ is contained in such a subspace $Y$, the norm of $X$ satisfies the parallelogram identity. Hence $X$ is a Hilbert space.
	
	The conjecture was first proved in the real case for $n=2$ by Auerbach, Mazur and Ulam~\cite{AMU1935}. 
	Dvoretzky's theorem~\cite{Dvoretzky1959} implies an affirmative answer for every $n\geq2$ when $X$ is an
	infinite-dimensional real Banach space. Milman~\cite{Milman1971} extended Dvoretzky's theorem to complex Banach spaces, thereby settling the infinite-dimensional complex case as well. 
	In 1967, Gromov \cite{Gromov1967} proved the conjecture for every even $n$, over both the real and complex fields; for odd $n$, he also proved it for real spaces when $\dim X\geq n+2$ and for complex spaces when $\dim X\geq2n$.
	
	Thus the remaining finite-dimensional real problem was concentrated in odd dimension and codimension one. Bor, Hern\'andez-Lamoneda, Jim\'enez-Desantiago and Montejano \cite{BorEtAl2021} settled the cases $n\equiv1\pmod4$, $n\geq5$, with the possible exception of $n=133$. 
	Ivanov, Mamaev and Nordskova~\cite{IMN2023} later settled $n=3$, $\dim X=4$. 
	In the complex case, Bracho and Montejano~\cite{BrachoMontejano2021} proved the conjecture for $n\equiv1\pmod4$. 
	
	The present paper is concerned with the real case. We prove Banach's isometric conjecture for every odd $n$, including all previously
	unresolved cases. Together with the earlier results, this settles Banach's isometric conjecture in the real case.
	
	\begin{theorem}\label{thm:main}
		Let $X$ be a real Banach space. If, for some fixed integer $2\leq n<\dim X$, all $n$-dimensional linear subspaces of $X$ are linearly isometric, then $X$ is a Hilbert space.
	\end{theorem}
	
	For finite-dimensional spaces, passing to unit balls gives the equivalent convex-geometric formulation.
	
	\begin{theorem}\label{cor:convex-main}
		Let $K$ be an origin-symmetric convex body in $\R^N$ and $2\leq n<N$. 
		If all $n$-dimensional central sections of $K$ are linearly equivalent, then $K$ is an ellipsoid.
	\end{theorem}
	
	In view of the codimension-one reduction, the odd-dimensional case of Theorem~\ref{thm:main} follows from the following hyperplane theorem.
	
	\begin{theorem}\label{thm:odd-hyperplane}
		Let $n\geq3$ be odd and let $K$ be an origin-symmetric convex body in $\R^{n+1}$. If all central hyperplane sections of $K$ are linearly equivalent, then $K$ is an ellipsoid.
	\end{theorem}
	
	Our proof shares its bundle-theoretic starting point with Gromov~\cite{Gromov1967} and Bor et al.~\cite{BorEtAl2021}.
	Fix an $n$-dimensional space $E$ and a model section $S\subset E$.
As $u$ varies over $\Sph^n$, the isomorphisms
$$
A:E\longrightarrow u^\perp,
\qquad
A(S)=K\cap u^\perp,
$$
form a principal $\Aut(S)$-bundle over $\Sph^n$.
Gromov~\cite{Gromov1967} uses topological restrictions on reductions of the tautological bundle over a Grassmannian to obtain the Euclidean conclusion.
Bor et al.~\cite{BorEtAl2021} combine restrictions on reductions of the tangent frame bundle of $\Sph^n$ with compact Lie group theory to show that $S$ is linearly equivalent to a symmetric body of revolution. 
Both approaches require detailed information about the compact group actions permitted by the relevant bundle reductions. In contrast to these bundle-theoretic approaches, Ivanov, Mamaev and Nordskova~\cite{IMN2023} treat the case $n=3$ by a local differential-geometric argument.
	
	We use the bundle differently. We neither classify $\Aut(S)$ nor try to produce additional symmetries of $S$. 
	After reducing the structure group to $\Aut(S)^\circ$, the bundle class lies in	$\pi_{n-1}(\Aut(S)^\circ)$, which is finite since $n$ is odd.
	Choose a positive integer $d$ that annihilates the bundle class and a smooth self-map $\phi:\Sph^n\to\Sph^n$ of degree $d$. The pullback bundle is then topologically trivial. A Lipschitz regularization argument produces a global Lipschitz family of exact maps $A_z$. See  Theorem~\ref{cor:global-exact-family}. 
    
	We next turn to degree theory to show that $S$ is an ellipsoid.
    For every $y\neq0$, exactness places the map $z\mapsto A_z y$ on the level set $p_S(y)\partial K$. Since $A_z y\perp\phi(z)$, its normalized direction
    $$
    z\longmapsto\frac{A_z y}{|A_z y|}
    $$
    is homotopic to $\phi$ and therefore has degree $\degree\phi$. 
    Applying the signed degree formula to the homogeneous extension of $z\mapsto A_z y$ shows that $y\mapsto p_S(y)^{n+1+2k}$ is a polynomial for every $k\geq0$ as proved in Lemma~\ref{prop:polynomial-powers}. The cases $k=0,1$, together with unique factorization, imply that $p_S^2$ is quadratic.
    Hence $S$ is an ellipsoid. This yields the main results.
	
	This article is organized as follows. Section~\ref{sec:background} collects some basic facts on convex bodies, principal	bundles and degree theory. Section~\ref{sec:finite-moment-method} constructs and globalizes the section isometry bundle.
	Section~\ref{sec:degree-conclusion} develops the degree argument, proves the polynomial rigidity of the model norm and concludes the main results.

	\vskip 25pt
	\section{\bf Preliminaries}
	\label{sec:background}
	\vskip 5pt
	
	Except for the Banach space $X$ in Theorem~\ref{thm:main}, which may be infinite-dimensional, all vector spaces below are real and finite-dimensional. Whenever Euclidean notions are used, the spaces involved are equipped with fixed auxiliary inner products and, when needed, fixed orientations. Norms denoted by $|\cdot|$, orthogonality, and adjoints are taken with respect to these inner products, while determinants and degrees are computed with respect to the chosen orientations.
	
	On $\R^n$ we use the standard Euclidean structure and orientation, and write $\Ball^n$ and $\Sph^{n-1}$ for the unit ball and unit sphere. Let $\kappa_n$ denote the Euclidean volume of $\Ball^n$. For an $n$-dimensional Euclidean vector space $E$, we write $\Ball(E)$ and $\Sph(E)$ for its unit ball and unit sphere, let $\vol_n$ denote the induced Lebesgue measure, and let $\dd\theta$ denote surface measure on $\Sph(E)$. For a linear subspace $H\subset E$, we write $H^\perp$ for its orthogonal complement. We write $I_E$ for the identity operator on $E$, and $\GL(E)$ and $\OO(E)$ for its general linear and orthogonal groups.
	
	For finite-dimensional real vector spaces $E$ and $V$, we write $\Hom(E,V)$ for the vector space of linear maps from $E$ to $V$ and set $\End(E)=\Hom(E,E)$. Using the auxiliary inner products fixed above, we equip $\Hom(E,V)$ with the Frobenius inner product
    $$
        \langle A,B\rangle=\operatorname{tr}(A^*B),\qquad A,B\in\Hom(E,V).
    $$
    Here $A^*:V\to E$ denotes the adjoint of $A$. We denote the corresponding norm by $\|\cdot\|$.

    For a Lipschitz map $F$ between metric spaces, $\Lip(F)$ denotes its optimal Lipschitz constant.
	
	\subsection{Convex bodies}
	\label{sec:convex-moments}
	\ 
	
	A \emph{convex body} in a vector space $E$ is a compact convex set with nonempty interior. For a convex body $K$, we write $\Int K$ and $\partial K$ for its interior and boundary, respectively. Good references on convex bodies are the books by Gardner \cite{Gardner1} and Schneider \cite{Sch}.
	
	Let $K\subset E$ be an origin-symmetric convex body in a Euclidean vector space. Its \emph{Minkowski functional} (or \emph{gauge}) and \emph{radial function} are defined by
	\begin{align*}
		p_K(x)&=\inf\{t>0:x\in tK\},\quad\ \ \, x\in E,\\
		\rho_K(\theta)&=\max\{r\geq0:r\theta\in K\},\quad \theta\in\Sph(E).
	\end{align*}
	The gauge $p_K$ is a norm, and the body and its boundary are recovered as
	$$K=\{x:p_K(x)\leq1\},\qquad\partial K=\{x:p_K(x)=1\}.$$
	On $\Sph(E)$, $\rho_K=1/p_K$.
	
	Let $K\subset E$ and $L\subset V$ be origin-symmetric convex bodies. If $T:E\to V$ is a linear isomorphism, then $p_{TK}(Tx)=p_K(x)$ for every $x\in E$. Thus $T(K)=L$ if and only if $T:(E,p_K)\to(V,p_L)$ is a linear isometry. If $H\subset E$ is a linear subspace, then
    \begin{equation}\label{eq:section-gauge}
        p_{K\cap H}(x)=p_K(x),\qquad x\in H.
    \end{equation}
	
	\begin{lemma}\label{lem:convex-estimates}
		If $K\subset E$ is an origin-symmetric convex body, then $\rho_K$ is Lipschitz on $\Sph(E)$.
	\end{lemma}
	
	\begin{proof}
    Choose $0<a\leq b$ such that $a\Ball(E)\subset K\subset b\Ball(E)$. These inclusions imply
    $b^{-1}|x|\leq p_K(x)\leq a^{-1}|x|$ for every $x\in E$. Hence, for $\theta,\eta\in\Sph(E)$, the reverse triangle inequality gives
    $$
        |\rho_K(\theta)-\rho_K(\eta)|=\frac{|p_K(\theta)-p_K(\eta)|}{p_K(\theta)p_K(\eta)}
        \leq b^2p_K(\theta-\eta)\leq\frac{b^2}{a}|\theta-\eta|.
    $$
    Thus $\rho_K$ is Lipschitz on $\Sph(E)$.
    \end{proof}
	
	We next introduce the moment tensors used in Section~\ref{sec:finite-moment-method}. Fix an $n$-dimensional Euclidean vector space $E$. For $j\geq1$, let $E^{\otimes j}$ denote the $j$-fold tensor product of $E$, equipped with the tensor-product inner product, and let $\Sym^j(E)$ denote the subspace of $E^{\otimes j}$ fixed by all permutations of the tensor factors. For $x\in E$, we write
    $$
        x^{\otimes j}=\underbrace{x\otimes\cdots\otimes x}_{j\text{ factors}}\in\Sym^j(E),
        \qquad\|x^{\otimes j}\|=|x|^j.
    $$
    
	For an origin-symmetric convex body $S\subset E$, its \emph{normalized $j$th moment tensor} is
    $$
        M_j(S)=\frac{1}{\vol_n(S)}\int_S x^{\otimes j}\dd x\in\Sym^j(E).
    $$
    The symmetry $S=-S$ implies that $M_j(S)=0$ for odd $j$. Under the canonical isometric identification of $\Sym^2(E)$ with the self-adjoint operators on $E$, $M_2(S)$ corresponds to the covariance operator $C(S)$ characterized by
    $$
        \ip{C(S)u}{v}=\frac{1}{\vol_n(S)}\int_S\ip{x}{u}\ip{x}{v}\dd x,\qquad u,v\in E.
    $$
	
    For a linear map $T:E\to V$, let $T^{\otimes j}:E^{\otimes j}\to V^{\otimes j}$ denote its tensor power, defined on pure tensors by
    $$
        T^{\otimes j}(x_1\otimes\cdots\otimes x_j)=Tx_1\otimes\cdots\otimes Tx_j.
    $$
    It maps $\Sym^j(E)$ into $\Sym^j(V)$. For an isomorphism $T$, change of variables gives
    \begin{equation}\label{eq:moment-equivariance}
        M_j(TS)=T^{\otimes j}M_j(S),\qquad
        C(TS)=TC(S)T^*.
    \end{equation}
	
	\begin{lemma}\label{lem:moment-lipschitz}
        Let $j\geq1$ and $0<a\leq b$. If the origin-symmetric convex bodies $S_1,S_2\subset E$ satisfy $a\Ball(E)\subset S_1,S_2\subset b\Ball(E)$, then
        $$
            \|M_j(S_1)-M_j(S_2)\|\leq\frac{2nb^{n+j-1}}{a^n}\|\rho_{S_1}-\rho_{S_2}\|_\infty.
        $$
    \end{lemma}
	
	\begin{proof}
		The assumed inclusions imply $a\leq\rho_{S_i}(\theta)\leq b$ for $i=1,2$ and $\theta\in\Sph(E)$. For $s,t\in[a,b]$ and every integer $m\geq1$, the mean value theorem gives $|s^m-t^m|\leq mb^{m-1}|s-t|$.

        Set $\delta=\|\rho_{S_1}-\rho_{S_2}\|_\infty$. For $i=1,2$, set $V_i=\vol_n(S_i)$ and $I_i=\int_{S_i}x^{\otimes j}\dd x$, so that $M_j(S_i)=I_i/V_i$. Using polar coordinates, the triangle inequality, and $\|\theta^{\otimes j}\|=1$ on $\Sph(E)$, we obtain
        \begin{align*}
            |V_1-V_2|&\leq\frac1n\int_{\Sph(E)}|\rho_{S_1}(\theta)^n-\rho_{S_2}(\theta)^n|\dd\theta\leq n\kappa_n b^{n-1}\delta,\\
            \|I_1-I_2\|&\leq\frac1{n+j}\int_{\Sph(E)}
            |\rho_{S_1}(\theta)^{n+j}-\rho_{S_2}(\theta)^{n+j}|\dd\theta\leq n\kappa_n b^{n+j-1}\delta.
        \end{align*}
        
		Finally, $V_i\geq a^n\kappa_n$ and $\|I_i\|\leq b^jV_i$ for $i=1,2$. The preceding estimates give
        \begin{align*}
            \|M_j(S_1)-M_j(S_2)\|\leq\frac{\|I_1-I_2\|}{V_1}
            +\frac{\|I_2\|\,|V_1-V_2|}{V_1V_2}\leq\frac{2nb^{n+j-1}}{a^n}\,\delta,
        \end{align*}
        as desired.
	\end{proof}
	
	For $j=2$, the lemma gives Lipschitz dependence of the covariance operator on the radial function under common inner and outer radius bounds.
	
	\begin{lemma}\label{lem:covariance-normalization}
        Let $S\subset E$ be an origin-symmetric convex body. Then $C(S)$ is positive definite, and $C\bigl(C(S)^{-1/2}S\bigr)=I_E$.
    \end{lemma}
	
	\begin{proof}
		For $v\neq0$, the set $\Int S\setminus v^\perp$ is nonempty and open. Hence
        \[
            \ip{C(S)v}{v}=\frac{1}{\vol_n(S)}\int_S\ip{x}{v}^2\dd x>0,
        \]
        so $C(S)$ is positive definite.

        Every positive-definite self-adjoint operator has a unique positive-definite square root~\cite[Chapter~1, item~(v)]{Bhatia2007}. Thus $C(S)^{-1/2}$ is well defined. By \eqref{eq:moment-equivariance},
        \[
            C\bigl(C(S)^{-1/2}S\bigr)=C(S)^{-1/2}C(S)C(S)^{-1/2}
            =I_E,
        \]
        as desired.
	\end{proof}
	
	\subsection{Principal bundles and classifying spaces}
	\label{sec:bundle-background}
	\ 
	
	We use right principal bundles throughout. One can refer to Husemoller~\cite{Husemoller1994} for more details.
    
	\begin{definition}\label{def:principal-bundle-conventions}
        Let $H$ be a topological group and let $Y$ be a topological space. A \emph{right principal $H$-bundle over $Y$} consists of a topological space $P$, a continuous surjection $\pi:P\to Y$, and a continuous right $H$-action on $P$ preserving the fiber $\pi^{-1}(y)$ for every $y\in Y$, such that $Y$ admits an open cover $\{U_i\}$ and, for each $i$, an $H$-equivariant homeomorphism
        \[
            \Phi_i:U_i\times H\longrightarrow\pi^{-1}(U_i),
            \qquad
            \pi(\Phi_i(y,h))=y.
        \]
        Here $H$ acts on $U_i\times H$ by right multiplication on the second factor. Such homeomorphisms are called \emph{local trivializations}.
    \end{definition}
	
	A \emph{local section} over an open set $U\subset Y$ is a continuous map $s:U\to P$ satisfying $\pi\circ s=\operatorname{id}_U$. A \emph{principal $H$-atlas} is a family $\{(U_i,s_i)\}$ such that $\{U_i\}$ is an open cover of $Y$ and each $s_i$ is a local section over $U_i$. Each $s_i$ determines a local trivialization $U_i\times H\to\pi^{-1}(U_i)$ given by $(y,h)\mapsto s_i(y)h$. On $U_i\cap U_j$, the transition function $t_{ji}:U_i\cap U_j\to H$ is uniquely determined by
    \[
        s_i(y)t_{ji}(y)=s_j(y).
    \]
    The bundle is \emph{trivial} if it is $H$-equivariantly isomorphic over $Y$ to $Y\times H$. Equivalently, it admits a global section.
	
	For a continuous map $f:Z\to Y$, the pullback bundle $f^*P\to Z$ is
    \[
        f^*P=\{(z,p)\in Z\times P:f(z)=\pi(p)\},
    \]
    with projection $(z,p)\mapsto z$ and right action $(z,p)h=(z,ph)$.

    If $H'\subset H$ is a closed subgroup, a reduction of the structure group to $H'$ is a principal $H'$-subbundle of $P$. If $W$ is a topological space with a continuous left $H$-action, the associated bundle is
    \[
        P\times_H W=(P\times W)/{\sim},
        \qquad
        (ph,w)\sim(p,hw).
    \]
    Since $\pi(ph)=\pi(p)$, the map $(p,w)\mapsto\pi(p)$ is constant on equivalence classes and hence induces the bundle projection $P\times_H W\to Y$.
	
	A \emph{classifying space} for $H$, denoted by $BH$, is the base of a universal principal $H$-bundle $EH\to BH$ whose total space $EH$ is contractible. If $Y$ is a CW complex, every principal $H$-bundle $P\to Y$ is isomorphic to $c^*EH$ for a map $c:Y\to BH$ that is unique up to free homotopy~\cite[Chapter~4, Sections~12--13]{Husemoller1994}. Such a map is called a classifying map for $P$.
	
	\begin{lemma}\label{lem:bundle-classification}
        Let $H$ be a path-connected Lie group and $k\geq2$. Principal $H$-bundles over $\Sph^k$ are classified by $\pi_k(BH)\cong\pi_{k-1}(H)$, with the trivial bundle corresponding to zero.
    \end{lemma}
	
	\begin{proof}
        By the homotopy exact sequence of $H\to EH\to BH$~\cite[Chapter~1, Theorem~5.3 and Chapter~4, Section~11]{Husemoller1994} and the contractibility of $EH$, we have
        \[
            \pi_k(BH)\cong\pi_{k-1}(H),
            \qquad
            \pi_1(BH)\cong\pi_0(H)=0.
        \]
        Thus $BH$ is simply connected, so free homotopy classes of maps $\Sph^k\to BH$ are naturally identified with $\pi_k(BH)$. The classification property of $BH$ now gives the claimed classification. The constant classifying map represents zero, and its pullback is the product bundle.
    \end{proof}
	
	Classifying maps are natural under pullback. If $P\cong c^*EH$ and $f:Z\to Y$, then
    \[
        f^*P\cong f^*c^*EH\cong(c\circ f)^*EH.
    \]
    Thus $c\circ f$ is a classifying map for $f^*P$.
	
	We also need the following standard finiteness fact.
	
	\begin{lemma}\label{lem:even-homotopy-finite}
		If $H$ is a compact connected Lie group and $k>0$ is even, then
		$\pi_k(H)$ is finite.
	\end{lemma}
	
	\begin{proof}
        There is a finite covering homomorphism $T\times H_{\mathrm{ss}}\to H$, where $T$ is a torus and $H_{\mathrm{ss}}$ is compact, connected, simply connected and semisimple~\cite[Section~2.9]{BorelHirzebruch1958}. Since $k\geq2$, this covering induces an isomorphism on $\pi_k$, and $\pi_k(T)=0$. Hence
        \[
            \pi_k(H)\cong\pi_k(T\times H_{\mathrm{ss}})
            \cong\pi_k(T)\oplus\pi_k(H_{\mathrm{ss}})
            =\pi_k(H_{\mathrm{ss}}).
        \]
        The group $\pi_k(H_{\mathrm{ss}})$ is finite by Serre's finiteness theorem in even degrees~\cite[Chapter~V, Section~3, Corollary~2]{Serre1953}.
    \end{proof}
	
	\subsection{Degree theory}
    \label{sec:degree-background}\

	In this part, we recall some facts on Brouwer degree. One can refer to Fonseca and Gangbo~\cite{FonsecaGangbo1995} for more details.
    \begin{definition}\label{def:degree-conventions}
        Let $\Omega\subset\R^n$ be bounded and open, let $F:\overline\Omega\to\R^n$ be continuous, and let $v\notin F(\partial\Omega)$. We write $\degree(F,\Omega,v)$ for the \emph{Brouwer degree} of $F$ on $\Omega$ at $v$. If $F$ is $C^1$ on $\Omega$ and $v$ is a regular value, then
        \[
            \degree(F,\Omega,v)
            =\sum_{x\in F^{-1}(v)}\operatorname{sgn}\det DF(x).
        \]
        We orient every sphere as the boundary of its Euclidean unit ball. For $n\geq2$ and a continuous map $h:\Sph^{n-1}\to\Sph^{n-1}$, we write $\degree h$ for the integer determined by
        \[
            h_*[\Sph^{n-1}]=(\degree h)[\Sph^{n-1}]
        \]
        in top homology.
    \end{definition}
	
	For Brouwer degree, we will use homotopy invariance, vanishing outside the image, the composition rule, and invariance under orientation-preserving changes of coordinates~\cite[Chapters~1-2]{FonsecaGangbo1995}. For sphere degree, we will use homotopy invariance and multiplicativity~\cite[Section~2.2]{Hatcher2002}.
	
	\begin{lemma}\label{lem:degree-precomposition}
        Let $n\geq2$, let $\phi:\Sph^n\to\Sph^n$ be a based map of degree $d$, and let $Y$ be a based space. Then, for every based map $f:\Sph^n\to Y$,
        \[
            [f\circ\phi]=d[f]
            \qquad\text{in }\pi_n(Y).
        \]
    \end{lemma}

    \begin{proof}
        Under the degree isomorphism $\pi_n(\Sph^n)\cong\mathbb Z$, one has $[\phi]=d[\operatorname{id}_{\Sph^n}]$~\cite[Section~4.1 and Corollary~4.25]{Hatcher2002}. Since $f_*$ is a homomorphism, $[f\circ\phi]=f_*[\phi]=d f_*[\operatorname{id}_{\Sph^n}]=d[f]$.
    \end{proof}
	
	\begin{lemma}\label{lem:radial-cone-degree}
        Let $n\geq2$ and let $h:\Sph^{n-1}\to\Sph^{n-1}$ be continuous. Define $C_h:\Ball^n\to\Ball^n$ by
        \[
            C_h(0)=0,
            \qquad
            C_h(t\theta)=th(\theta),
            \quad 0<t\leq1,\quad \theta\in\Sph^{n-1}.
        \]
        Then for each $v\in\Int\Ball^n$,
        \[
            \degree(C_h,\Int\Ball^n,v)=\degree h.
        \]
    \end{lemma}
	
	\begin{proof}
        Let $v\in\Int\Ball^n$. By the boundary characterization of Brouwer degree \cite[Proposition~1.27]{FonsecaGangbo1995},
        \[
            \degree(C_h,\Int\Ball^n,v)
            =\degree\left(\frac{h-v}{|h-v|}\right).
        \]
        The maps $\theta\mapsto(h(\theta)-tv)/|h(\theta)-tv|$, $0\leq t\leq1$, form a homotopy from $h$ to $(h-v)/|h-v|$, since $|h(\theta)-tv|\geq1-t|v|>0$. Hence the right-hand side equals $\degree h$.
    \end{proof}
	
	\begin{lemma}\label{prop:signed-degree-formula}
        Let $D\subset\R^m$ be open and let $\Omega\Subset D$ be open with $\vol_m(\partial\Omega)=0$. If $F:D\to\R^m$ is Lipschitz, then for each bounded Borel function $\varphi:\R^m\to\R$,
        \[
            \int_\Omega \varphi(F(x))\det DF(x)\dd x
            =\int_{\R^m}\varphi(v)\degree(F,\Omega,v)\dd v.
        \]
    \end{lemma}
	
	\begin{proof}
        Choose a bounded open set $D_0$ such that $\overline\Omega\subset D_0\Subset D$. By Rademacher's theorem~\cite[Theorem~3.2]{EvansGariepy2015}, $F$ is differentiable almost everywhere in $D_0$ and $|\det DF|\leq\Lip(F)^m$ almost everywhere. Hence $\det DF\in L^1(D_0)$. The area formula~\cite[Theorem~3.8]{EvansGariepy2015} implies that Lipschitz maps send null sets to null sets, so $\vol_m(F(\partial\Omega))=0$. Applying the change-of-variables formula with degree~\cite[Remark~5.26(ii) and Theorem~5.27]{FonsecaGangbo1995} to $F|_{D_0}$ and $\Omega$ gives the desired identity.
\end{proof}
	
	\vskip 25pt
	\section{\bf The isometry bundle and its globalization}
	\label{sec:finite-moment-method}
	\label{sec:global-exact-construction}
	\vskip 5pt
	
	The main purpose of this section is to prove  Theorem \ref{cor:global-exact-family}. We first construct finite-moment coordinates that detect the exact symmetries of a model body. We then apply those coordinates to the moving hyperplane sections, assemble the resulting local exact maps into a Lipschitz principal bundle, remove its topological obstruction, and regularize a trivializing section without losing exactness.
	
	\subsection{Finite-moment detection and orbit coordinates}
	\label{sec:moment-orbit}
	\ 
	
	Let $E$ be an $n$-dimensional Euclidean vector space and let $S\subset E$ be an origin-symmetric convex body. By Lemma~\ref{lem:covariance-normalization}, $C(S)$ is positive definite. Replacing $S$ by the linearly equivalent body $C(S)^{-1/2}S$ changes none of the questions at issue and gives the normalization
	\begin{equation}\label{eq:covariance-normalization}
		C(S)=I_E.
	\end{equation}
	Set
	$$
	G=\Aut(S)=\{g\in\GL(E):gS=S\}.
	$$
	
	\begin{lemma}\label{lem:finite-moments}
		The group $G$ is a compact subgroup of $\OO(E)$. Moreover, there exist positive integers $j_1,\ldots,j_s$ such that $G$ is exactly the stabilizer in $\OO(E)$ of the finite tuple
		$$
		\mathbf M(S)=\bigl(M_{j_1}(S),\ldots,M_{j_s}(S)\bigr).
		$$
	\end{lemma}
	
	\begin{proof}
		If $g\in G$, then \eqref{eq:moment-equivariance} and
		\eqref{eq:covariance-normalization} give
		$$
		I_E=C(S)=C(gS)=gC(S)g^*=gg^*.
		$$
		Hence $g\in\OO(E)$. The group $G$ is closed in $\OO(E)$ and is therefore compact.
		
		It remains to prove the finite-moment assertion. We divide the argument into two steps.
		
		\textbf{Step 1.} Show that $G$ is the common stabilizer of all moment tensors.
		
		By \eqref{eq:moment-equivariance}, every $g\in G$ satisfies $g^{\otimes j}M_j(S)=M_j(S)$ for each $j\geq 1$. That is, $G$ is contained in the common stabilizer in $\OO(E)$ of all the moment tensors of $S$. 
		
		For the reverse inclusion, assume $Q\in\OO(E)$ preserves every moment tensor of $S$. By \eqref{eq:moment-equivariance}, the normalized uniform measures on $S$ and $QS$ then have the same integral against every polynomial. By the Stone--Weierstrass theorem~\cite[Theorem~7.32]{Rudin1976}, polynomials are uniformly dense in $C(S\cup QS)$. The uniqueness of finite Radon measures from their integrals against continuous functions therefore shows that the two measures agree. Their supports are respectively $S$ and $QS$, so $QS=S$ and therefore $Q\in G$.
		
		\textbf{Step 2.} Reduce the infinite family of moment equations to finitely many of them.
		
		Choose matrix coordinates $Q=(q_{ab})_{a,b=1}^n$ on
		$\OO(E)\subset\R^{n^2}$. Every coordinate of
		$$Q^{\otimes j}M_j(S)-M_j(S)$$
		is a polynomial in the variables $q_{ab}$. Let $I$ be the ideal generated by all these coordinate polynomials. Since $\R[q_{ab}:1\leq a,b\leq n]$ is Noetherian~\cite[Corollary~7.6]{AtiyahMacdonald1969}, we can choose finite generators $f_1,\ldots,f_r$ of $I$. Each $f_i$ is a finite polynomial combination of these coordinate polynomials. If $J$ is generated by the coordinate polynomials occurring in these finitely many combinations, then
		$$
		(f_1,\ldots,f_r)\subset J\subset I=(f_1,\ldots,f_r),
		$$
		so $J=I$. Grouping these finitely many coordinate equations by tensor degree gives $j_1,\ldots,j_s$. The corresponding moment tensors therefore have common stabilizer $G$ by Step~1.
	\end{proof}
	
	Let
	$$
	\mathcal W=\bigoplus_{i=1}^s\Sym^{j_i}(E),
	$$
	equipped with the Euclidean direct-sum inner product induced by that of $E$. Then $\mathbf M(S)\in\mathcal W$. We denote its $\OO(E)$-orbit in
	$\mathcal W$ by $\OO(E)\mathbf M(S)$.
	
	\begin{lemma}\label{prop:moment-orbit}
		For the finite tuple from Lemma~\ref{lem:finite-moments}, the orbit map
		$$
		\Theta:\OO(E)/G\longrightarrow\OO(E)\mathbf M(S),
		\qquad
		\Theta(QG)=Q\mathbf M(S),
		$$
		is a diffeomorphism onto a compact embedded submanifold of $\mathcal W$. Its inverse is locally Lipschitz with respect to the Euclidean distance of $\mathcal W$.
	\end{lemma}
	
	\begin{proof}
		Lemma~\ref{lem:finite-moments} identifies $G$ with the stabilizer of $\mathbf M(S)$ in $\OO(E)$. By the closed subgroup theorem and the quotient manifold theorem, $\OO(E)/G$ is a smooth compact manifold, and $\Theta$ is a smooth bijection onto the orbit~\cite[Theorems~20.12 and 21.10]{Lee2013}.
		
		We now show that $\Theta$ is an embedding. Let $\mathfrak o(E)$ and $\mathfrak g$ be the Lie algebras of $\OO(E)$ and $G$, respectively, and let $\rho$ be the direct-sum tensor representation on $\mathcal W$. Thus
		$$
		\rho(Q)(T_i)_{i=1}^s=(Q^{\otimes j_i}T_i)_{i=1}^s,
		\quad\text{and}\quad
		\Theta(QG)=\rho(Q)\mathbf M(S).
		$$
		Using $T_G(\OO(E)/G)\cong\mathfrak o(E)/\mathfrak g$, we obtain
		$$d\Theta_G([\xi])=d\rho(\xi)\mathbf M(S),
		\qquad \xi\in\mathfrak o(E).$$
		If $d\Theta_G([\xi])=0$, then
		$$\rho(\exp(t\xi))\mathbf M(S)=\exp\bigl(t\,d\rho(\xi)\bigr)\mathbf M(S)=\mathbf M(S)$$
		for every $t$. Hence $\exp(t\xi)\in G$  and therefore $\xi\in\mathfrak g$. Thus $d\Theta_{G}$ is injective. The equivariance identity
		$$
		\Theta\circ L_Q=\rho(Q)\circ\Theta,
		$$
		where $L_Q$ denotes left translation, then gives injectivity at every $QG$. Consequently, $\Theta$ is an injective immersion. Since $\OO(E)/G$ is compact and $\mathcal W$ is Hausdorff, $\Theta$ is an embedding and hence a diffeomorphism onto its orbit.
		
		It remains to prove the metric assertion. Fix a Riemannian distance $d_{\OO(E)/G}$. Given $m\in\OO(E)\mathbf M(S)$, write $\OO(E)\mathbf M(S)$ near $m$ as a smooth graph
		$$
		\Gamma(v)=m+v+\psi(v),
		\qquad v\in U_0\subset T_m(\OO(E)\mathbf M(S)).
		$$
		Choose a convex neighborhood $U\Subset U_0$ of $0$. For $v,v'\in U$, put $x=\Gamma(v)$ and $x'=\Gamma(v')$. Orthogonal projection onto $T_m(\OO(E)\mathbf M(S))$ gives $|v-v'|\leq|x-x'|$. Since $F=\Theta^{-1}\circ\Gamma$ is smooth on $U_0$,
		$$
		C=\max_{\overline U}\|dF\|<\infty.
		$$
		Integration along the segment from $v$ to $v'$ gives
		$$
		d_{\OO(E)/G}\bigl(\Theta^{-1}(x),\Theta^{-1}(x')\bigr)
		\leq C|v-v'|
		\leq C|x-x'|.
		$$
		Thus $\Theta^{-1}$ is locally Lipschitz for the ambient distance of $\mathcal W$.
	\end{proof}
	
	The finite tuple now contains all the symmetry information needed below. Lemma~\ref{lem:finite-moments} identifies $G=\Aut(S)$ as its exact stabilizer, while Lemma~\ref{prop:moment-orbit} shows that the corresponding coset in $\OO(E)/G$ is a locally Lipschitz function of the tuple. In the next subsection, we will apply these facts to a family of linearly equivalent hyperplane sections. The goal is to construct local exact maps from a fixed model section to the varying sections and then assemble these maps into a Lipschitz principal $G$-bundle.
	
	\subsection{The local isometry bundle}
    \
	
	All Lipschitz statements below use fixed Riemannian distances on the base manifolds, restricted to subsets when necessary.
	
	\begin{definition}\label{def:lipschitz-principal-bundle}
		Let $P\to Y$ be a principal $G$-bundle, where $Y$ is compact and smooth and $G$ is a compact matrix group. A finite principal $G$-atlas $\{(U_i,s_i)\}$ is called \emph{Lipschitz} if every transition function
		$$
		t_{ji}:U_i\cap U_j\longrightarrow G,
		\qquad
		s_i(x)t_{ji}(x)=s_j(x),
		$$
		is Lipschitz on the whole overlap. A section $s$ is called Lipschitz if its coordinate maps $g_i:U_i\to G$, determined by $s|_{U_i}=s_i g_i$, are Lipschitz.
	\end{definition}
	
	For the remainder of this section, let $n\geq2$ and let $K\subset\R^{n+1}$ be an origin-symmetric convex body whose central hyperplane sections are pairwise linearly equivalent.  Fix $u_0\in\Sph^n$, put $E=u_0^\perp$, and choose a model section $S\subset E$ linearly equivalent to every section $S_u=K\cap u^\perp$. Normalize $S$ as above, so that $C(S)=I_E$. Recall that $G=\Aut(S)\subset\OO(E)$.
	
	For $u\in\Sph^n$, let
	$$
	\cP_u=\{A\in\Hom(E,\R^{n+1}):A(E)=u^\perp,\ A(S)=S_u\}.
	$$
	The group $G$ acts freely and transitively on $\cP_u$ by right composition.
	
	We now construct exact local sections of the disjoint union $\cP=\bigsqcup_{u\in\Sph^n}\cP_u$.
	
	\begin{lemma}\label{prop:local-exact-atlas}
		The family $\cP\to\Sph^n$ admits a principal $G$-bundle structure with a finite Lipschitz principal $G$-atlas. More precisely, every $u_*\in\Sph^n$ has a connected neighborhood $U$ and a Lipschitz map $u\mapsto A_u\in\Hom(E,\R^{n+1})$ satisfying
		$$
		A_u(E)=u^\perp
		\quad\text{and}\quad
		A_u(S)=S_u.
		$$
		The transition functions of a finite family of such charts are globally Lipschitz on their overlaps.
	\end{lemma}
	
	\begin{proof}
	    We divide the proof into three steps.
	
	    \textbf{Step 1.} Pull the varying sections back to $E$ and normalize their covariance.
	
	    Fix $u_*\in\Sph^n$ and choose a Euclidean isometry $R_*:E\to u_*^\perp$. For $u\in\Sph^n$, let $\Pi_u$ be the orthogonal projection of $\R^{n+1}$ onto $u^\perp$. Since $R_*^*\Pi_{u_*}R_*=I_E$, choose a connected neighborhood $U$ of $u_*$ with compact closure such that $R_*^*\Pi_uR_*$ is positive definite on a neighborhood of $\overline U$. For $u\in U$, set
	    $$
	    J_u=\Pi_u R_*\bigl(R_*^*\Pi_u R_*\bigr)^{-1/2}.
	    $$
	    Moreover,
	    $$
        J_u^*J_u=\bigl(R_*^*\Pi_uR_*\bigr)^{-1/2}R_*^*\Pi_uR_*\bigl(R_*^*\Pi_uR_*\bigr)^{-1/2}=I_E.
	    $$
	    Thus $J_u:E\to u^\perp$ is a Euclidean isometry and $u\mapsto J_u$ is smooth on a neighborhood of $\overline U$ and hence Lipschitz on $U$.
	
	    Define the pulled-back section $\widetilde S_u=J_u^{-1}S_u$. For $\theta\in\Sph(E)$, we have $\rho_{\widetilde S_u}(\theta)=\rho_K(J_u\theta)$. The radial function of $K$ is Lipschitz by Lemma~\ref{lem:convex-estimates}, so this identity shows that $u\mapsto\rho_{\widetilde S_u}$ is Lipschitz in the uniform norm. Since each $J_u$ is an isometry, the bodies $\widetilde S_u$ have common inner and outer radii. Lemma~\ref{lem:moment-lipschitz} therefore implies that $C_u=C(\widetilde S_u)$ depends Lipschitz-continuously on $u$.
	
	    By Lemma~\ref{lem:covariance-normalization}, each $C_u$ is positive definite. Since $\overline U$ is compact, the operators $C_u$ remain in a compact subset of the open cone of positive-definite self-adjoint operators on $E$. Since the maps $T\mapsto T^{\pm1/2}$ are smooth on this cone, $u\mapsto C_u^{\pm1/2}$ is Lipschitz.
	
	    Put $\widehat S_u=C_u^{-1/2}\widetilde S_u$. By Lemma~\ref{lem:covariance-normalization}, $C(\widehat S_u)=I_E$. For each fixed $u$, the body $\widehat S_u$ is linearly equivalent to $S$, so there exists $T\in\GL(E)$ such that $\widehat S_u=TS$. By \eqref{eq:moment-equivariance} and \eqref{eq:covariance-normalization},
	    \[
	    I_E=C(\widehat S_u)=TC(S)T^*=TT^*.
	    \]
	    Thus $T\in\OO(E)$, and hence $\widehat S_u\in\OO(E)S$.
	
	    \textbf{Step 2.} Lift the moment-orbit coordinates to Lipschitz orthogonal maps.
	
	    Let $\mathbf M$ be the finite moment tuple from Lemma~\ref{lem:finite-moments}. For every selected degree $j$, \eqref{eq:moment-equivariance} gives
	    \[
	    M_j(\widehat S_u)=(C_u^{-1/2})^{\otimes j}M_j(\widetilde S_u).
	    \]
	    By Step~1, $u\mapsto C_u^{-1/2}$ is Lipschitz and bounded, while Lemma~\ref{lem:moment-lipschitz} shows that $u\mapsto M_j(\widetilde S_u)$ is Lipschitz. Hence the right-hand side depends Lipschitz-continuously on $u$. Step~1 also gives $\widehat S_u\in\OO(E)S$, and hence $\mathbf M(\widehat S_u)\in\OO(E)\mathbf M(S)$. Therefore $u\mapsto\mathbf M(\widehat S_u)$ is a Lipschitz map into the orbit $\OO(E)\mathbf M(S)$.
	
	    By Lemma~\ref{prop:moment-orbit}, the map
$$
u\longmapsto
\Theta^{-1}\bigl(\mathbf M(\widehat S_u)\bigr)
$$
into $\OO(E)/G$ is locally Lipschitz. After shrinking $U$, its image lies in a relatively compact geodesically convex neighborhood $V$ on which the quotient map $\OO(E)\to\OO(E)/G$ admits a smooth local section $\sigma$. 
The section $\sigma$ is Lipschitz on $V$, and hence
$$
Q_u=
\sigma\!\big(\Theta^{-1}(\mathbf M(\widehat S_u))\big)
$$
defines a Lipschitz map $U\to\OO(E)$. By the definition of $\Theta$, the coset
$
\Theta^{-1}\bigl(\mathbf M(\widehat S_u)\bigr)
$
is the unique coset $QG$ such that $QS=\widehat S_u$.
Since $Q_u$ represents this coset, it follows that
$Q_uS=\widehat S_u$.

	    Undoing the normalization and returning to $u^\perp$, define
	    \[
	    A_u=J_u C_u^{1/2}Q_u.
	    \]
	    Then $u\mapsto A_u$ is Lipschitz and $A_u(E)=u^\perp$. Moreover,
	    $$
	    A_u(S)=J_uC_u^{1/2}\widehat S_u=J_u\widetilde S_u=S_u.
	    $$
	
	    \textbf{Step 3.} Assemble the local maps into a finite Lipschitz principal $G$-atlas.
	
	    Choose finitely many connected open sets $U_i$ covering $\Sph^n$, with each $\overline{U_i}$ contained in a larger neighborhood on which a local map $A_i$ constructed above is defined. On $U_i\cap U_j$, there is a unique $t_{ji}(u)\in G$ such that $A_j(u)=A_i(u)t_{ji}(u)$, and
	    $$
	    t_{ji}(u)=\bigl(A_i(u)^*A_i(u)\bigr)^{-1}A_i(u)^*A_j(u).
	    $$
	    The smallest singular value of $A_i(u)$ is bounded below uniformly for $u\in\overline{U_i}$. Consequently, $u\mapsto\bigl(A_i(u)^*A_i(u)\bigr)^{-1}$ is bounded and Lipschitz on $U_i$. All factors in the displayed formula are therefore bounded and Lipschitz on $U_i\cap U_j$, so $t_{ji}$ is globally Lipschitz there.
	
	    It remains to define the bundle topology on $\cP$. For each $i$, define
	    $$
	    \Phi_i:U_i\times G\longrightarrow\cP|_{U_i},
	    \qquad
	    \Phi_i(u,g)=\bigl(u,A_i(u)g\bigr).
	    $$
	    This map is bijective since $G$ acts freely and transitively on each fiber $\cP_u$. On $U_i\cap U_j$, the change of coordinates is
	    $$
	    \bigl(\Phi_i^{-1}\circ\Phi_j\bigr)(u,g)
	    =\bigl(u,t_{ji}(u)g\bigr).
	    $$
	    The coordinate changes are homeomorphisms, and the uniqueness of the transition functions gives the cocycle identity $t_{ki}=t_{ji}t_{kj}$ on every triple overlap. Hence the maps $\Phi_i$ define a topology on $\cP$ for which they are $G$-equivariant local trivializations. Thus $\cP\to\Sph^n$ is a principal $G$-bundle. Since the cover is finite and every $t_{ji}$ is Lipschitz on its whole overlap, the local sections $u\mapsto\bigl(u,A_i(u)\bigr)$ form the required finite Lipschitz principal $G$-atlas.
    \end{proof}
	
	Using the finite-moment orbit, Lemma~\ref{prop:local-exact-atlas} turns the pointwise linear equivalences into a Lipschitz principal $G$-bundle of exact maps. To obtain a global Lipschitz family of such maps, we next reduce the structure group and remove the topological obstruction to a global section by passing to a suitable pullback.
	
	\subsection{The globalization of the isometry bundle}\
	
	Let $G^\circ$ denote the connected component of $G$ containing the identity $I_E$. We first reduce the structure group from $G$ to $G^\circ$. Recall that 
	$$
	\cP_u=\{A\in\Hom(E,\R^{n+1}):A(E)=u^\perp,\ A(S)=S_u\},\quad u\in\Sph^n,
	$$
	and $\cP=\bigsqcup_{u\in\Sph^n}\cP_u$.
	
	\begin{lemma}\label{prop:identity-component}
		If $n\geq2$, then the principal $G$-bundle $\cP\to\Sph^n$ contains a Lipschitz principal $G^\circ$-subbundle $\cP^\circ\to\Sph^n$.
	\end{lemma}
	
	\begin{proof}
		The identity component $G^\circ$ is a closed normal subgroup of finite index in $G$. Indeed, connected components are closed, and conjugation preserves the component containing the identity, so $G^\circ$ is closed and normal. Since $G$ is a Lie group, $G^\circ$ is also open. Hence $G/G^\circ$ is discrete. It is compact since $G$ is compact, and is therefore finite. Consequently, $\cP/G^\circ\to\Sph^n$ is a principal $G/G^\circ$-bundle and hence a finite covering. Since $\Sph^n$ is simply connected for $n\geq 2$, every connected component of $\cP/G^\circ$ maps homeomorphically onto $\Sph^n$. 
		
		Choose a connected component $\Sigma\subset\cP/G^\circ$, and let $\cP^\circ\subset\cP$ be its preimage under the quotient map $\cP\to\cP/G^\circ$.
		Since $\Sigma$ maps homeomorphically onto $\Sph^n$, $\cP^\circ$ meets each fiber $\cP_u$ in exactly one $G^\circ$-orbit. Thus $\cP^\circ\to\Sph^n$ is a principal $G^\circ$-subbundle of $\cP\to\Sph^n$.
		
		To verify the Lipschitz assertion, take a finite Lipschitz atlas $\{(U_i,s_i)\}$ with connected chart domains, as supplied by Lemma~\ref{prop:local-exact-atlas}. The local section $s_i$ identifies $(\cP|_{U_i})/G^\circ$ with $U_i\times G/G^\circ$. Under this identification, $\Sigma$ is the graph of a continuous map $U_i\to G/G^\circ$. Since $U_i$ is connected and $G/G^\circ$ is discrete, this map is constant. Denote its value by $g_iG^\circ$. Then $s_i^\circ=s_i g_i$ is a local section of $\cP^\circ$. If $s_i t_{ji}=s_j$, then
		$$t_{ji}^\circ=g_i^{-1}t_{ji}g_j\in G^\circ$$
		and $s_i^\circ t_{ji}^\circ=s_j^\circ$. Since $g_i,g_j$ are constant, the reduced transition function is globally Lipschitz.
	\end{proof}
	
	The reduced bundle may still be nontrivial, but when $n\geq3$ is \emph{odd}, its class has finite order. Pulling it back along a suitable positive-degree self-map of $\Sph^n$ therefore yields a trivial bundle.
	
	\begin{lemma}\label{prop:torsion-pullback}
		If $n\geq3$ is odd, then there exists a smooth map $\phi:\Sph^n\to\Sph^n$ of degree $d\geq1$ such that $\phi^*\cP^\circ$ is topologically trivial.
	\end{lemma}
	
	\begin{proof}
		Fix basepoints $x_0\in\Sph^n$ and $b_0\in B(G^\circ)$, where $B(G^\circ)$ is the classifying space of $G^\circ$. By Lemma~\ref{lem:bundle-classification}, choose a based classifying map
        $$
            c:(\Sph^n,x_0)\longrightarrow\bigl(B(G^\circ),b_0\bigr)
        $$
        for $\cP^\circ$, and set
        $$
            \alpha=[c]\in\pi_n\bigl(B(G^\circ),b_0\bigr)\cong\pi_{n-1}(G^\circ).
        $$
		Since the group $G^\circ$ is compact and connected and $n-1$ is positive and even,  Lemma~\ref{lem:even-homotopy-finite} shows that $\pi_{n-1}(G^\circ)$ is finite. Choose $d\geq1$ such that $d\alpha=0$.
		
		Choose a continuous based representative of $d\in\pi_n(\Sph^n,x_0)\cong\mathbb Z$. Relative smooth approximation gives a smooth based representative $\phi:\Sph^n\to\Sph^n$, still of degree $d$. By naturality of classifying maps, $c\circ\phi$ classifies $\phi^*\cP^\circ$. Lemma~\ref{lem:degree-precomposition} gives
		$$
		[c\circ\phi]=d[c]=d\alpha=0.
		$$
		The zero class corresponds to the trivial principal $G^\circ$-bundle by Lemma~\ref{lem:bundle-classification}. Hence $\phi^*\cP^\circ$ is topologically trivial.
	\end{proof}
	
	The pullback bundle $\phi^*\cP^\circ$ is now topologically trivial and hence admits a global continuous section. The degree argument in the next section, however, requires such a section to be Lipschitz.
	The following approximation and retraction lemmas provide the required regularity upgrade.
	
	By a Euclidean vector bundle we mean a real vector bundle equipped with a continuous fiberwise inner product. An atlas for such a bundle is called orthonormal if each of its local trivializations is fiberwise isometric.
	\begin{lemma}\label{lem:vector-section-approximation}
		Let $\cE\to Y$ be a Euclidean vector bundle over a compact smooth manifold $Y$. If $\cE$ admits a finite orthonormal atlas with Lipschitz transition functions, then every continuous section can be uniformly approximated by Lipschitz sections.
	\end{lemma}
	
	\begin{proof}
		Let $\sigma_0$ be a continuous section and let $\varepsilon>0$. Choose a finite orthonormal atlas $\{(U_i,\tau_i)\}$, where $\tau_i:\cE|_{U_i}\to U_i\times V$ and $V$ is a fixed Euclidean vector space. Choose a smooth partition of unity $\{\psi_i\}$ such that $\supp\psi_i\Subset U_i$. Write $\sigma_0$ in the $i$th trivialization as a continuous map $f_i:U_i\to V$. By smooth approximation~\cite[Theorem~6.21]{Lee2013}, choose a smooth map $g_{i,\varepsilon}:U_i\to V$ satisfying
		$$
		|g_{i,\varepsilon}-f_i|<\varepsilon
		\qquad\text{on }\supp\psi_i.
		$$
		
		The local section represented by $\psi_i g_{i,\varepsilon}$ extends by zero to a global Lipschitz section $\eta_{i,\varepsilon}$. Indeed, it vanishes near $Y\setminus U_i$, and in every other chart its coordinates are the product of a Lipschitz orthogonal transition function and the compactly supported smooth map $\psi_i g_{i,\varepsilon}$.
		
		Set $\eta_\varepsilon=\sum_i\eta_{i,\varepsilon}$. Since the sum is finite, $\eta_\varepsilon$ is Lipschitz. Since the trivializations are orthonormal and $\sum_i\psi_i=1$, we have
		$$
		|\eta_\varepsilon(y)-\sigma_0(y)|\leq\sum_{i:\,y\in U_i}\psi_i(y)|g_{i,\varepsilon}(y)-f_i(y)|<\varepsilon
		$$
		for every $y\in Y$. Hence $\eta_\varepsilon\to\sigma_0$ uniformly as $\varepsilon\to0$.
	\end{proof}

	\begin{lemma}\label{lem:equivariant-retraction}
		Let $V$ be a finite-dimensional Euclidean space and let $G\subset\OO(V)$ be a compact subgroup. There exist an open neighborhood $\mathcal U\subset\End(V)$ of $G$, invariant under left multiplication by $G$, and a Lipschitz retraction $q:\mathcal U\to G$ such that
		$$
		q(gT)=gq(T)
		$$
		for every $g\in G$ and $T\in\mathcal U$.
	\end{lemma}
	
	\begin{proof}
		Since $G$ is compact, it is closed in $\OO(V)$. The closed subgroup theorem~\cite[Theorem~20.12]{Lee2013} therefore shows that $G$ is an embedded Lie subgroup of $\OO(V)$, and hence an embedded submanifold of $\End(V)$. For every $g\in G$, left multiplication $T\mapsto gT$ is a linear isometry of $\End(V)$ carrying $G$ onto itself. Hence, for every $h\in G$, it maps the tangent space $T_hG$ onto $T_{gh}G$ and therefore maps the normal space $(T_hG)^\perp$ onto $(T_{gh}G)^\perp$.
		
		Since $G$ is compact, the tubular neighborhood theorem~\cite[Theorem~6.24]{Lee2013} gives $\varepsilon>0$ so that
		$$
		\Xi:\{(h,\nu):h\in G,\ \nu\in(T_hG)^\perp,\ \|\nu\|<\varepsilon\}\longrightarrow\mathcal U_0,
		\qquad
		\Xi(h,\nu)=h+\nu,
		$$
		is a diffeomorphism onto an open neighborhood $\mathcal U_0$ of $G$. The domain of $\Xi$ is invariant under $(h,\nu)\mapsto(gh,g\nu)$, and
		$$
		\Xi(gh,g\nu)=g\Xi(h,\nu).
		$$
		Thus $\mathcal U_0$ is invariant under left multiplication by $G$, and the map $q_0:\mathcal U_0\to G$ defined by $q_0(\Xi(h,\nu))=h$ is a smooth retraction satisfying $q_0(gT)=gq_0(T)$.
		
		It remains to obtain a Lipschitz retraction on a smaller neighborhood. Choose $0<\varepsilon'<\varepsilon$ and let $\mathcal U$ be the image under $\Xi$ of the tube $\|\nu\|<\varepsilon'$. Then $\mathcal U$ is invariant under left multiplication by $G$, and
		$\overline{\mathcal U}$ is compactly contained in $\mathcal U_0$. Set $q=q_0|_{\mathcal U}$. Choose $\delta>0$ such that the closed $\delta$-neighborhood of $\overline{\mathcal U}$ lies in
		$\mathcal U_0$. Since this closed neighborhood is compact, the derivative of $q_0$ is bounded there by some constant $\Lambda$.
		
		If $T,T'\in\mathcal U$ and $\|T-T'\|<\delta$, the segment joining them lies in the closed $\delta$-neighborhood of $\overline{\mathcal U}$. Integrating the derivative of $q_0$ along this segment gives
		$$
		\|q(T)-q(T')\|\leq\Lambda\|T-T'\|.
		$$
		If $\|T-T'\|\geq\delta$, then $q(T),q(T')\in G$, and hence
		$$
		\|q(T)-q(T')\|\leq\operatorname{diam}(G)\leq\frac{\operatorname{diam}(G)}{\delta}\|T-T'\|.
		$$
		Thus $q$ is Lipschitz on $\mathcal U$, with Lipschitz constant at most $\max\{\Lambda,\operatorname{diam}(G)/\delta\}$.
	\end{proof}
	
	The preceding two lemmas combine to give the Lipschitz regularization needed for the topologically trivial pullback.
	
	\begin{theorem}\label{thm:lipschitz-section}
		Let $V$ be a finite-dimensional Euclidean space, and let $P\to Y$ be a principal $G$-bundle with a Lipschitz atlas, where $Y$ is a compact smooth manifold and $G\subset\OO(V)$ is a compact subgroup. If $P$
		is topologically trivial, then $P$ admits a global Lipschitz section.
	\end{theorem}
	
	\begin{proof}
		We divide the proof into two steps.
		
		\textbf{Step 1.} Approximate a continuous section inside an associated vector bundle.
		
		Form the associated vector bundle $\cE\to Y$ given by
		$$
		\cE=P\times_G\End(V),
		\qquad
		(pg,T)\sim(p,gT).
		$$
		Its transition functions act on $\End(V)$ by left multiplication. They are Lipschitz and orthogonal because the principal atlas is Lipschitz and $G\subset\OO(V)$. Define $\iota:P\longrightarrow\cE$ by $\iota(p)=[p,I_V]$.
		Since $\iota(pg)=[p,g]$, the map $\iota$ identifies $P$ with the orbit subbundle $P\times_G G\subset\cE$.
		
		Since $P$ is topologically trivial, it admits a continuous section $s_0:Y\to P$. Let $\mathcal U\subset\End(V)$ and $q:\mathcal U\to G$ be as in Lemma~\ref{lem:equivariant-retraction}. Since $G$ is compact and
		$\mathcal U$ is open, choose $\delta>0$ so that the $\delta$-neighborhood of $G$ lies in $\mathcal U$. By Lemma~\ref{lem:vector-section-approximation}, there is a Lipschitz
		section $\sigma$ of $\cE$ such that
		$$
		\|\sigma-\iota\circ s_0\|_\infty<\delta.
		$$
		Since $\iota\circ s_0$ takes values in $P\times_GG$, the estimate and the choice of $\delta$ show that $\sigma$ takes values in $P\times_G\mathcal U$.
		
		\textbf{Step 2.} Retract the approximation onto the orbit subbundle.
		
		Define
		$$
		Q:P\times_G\mathcal U\longrightarrow P\times_G G,
		\qquad
		Q([p,T])=[p,q(T)].
		$$
		This map is well defined. Indeed, for $g\in G$ and $T\in\mathcal U$,
		$$
		Q([pg,T])=[pg,q(T)]=[p,gq(T)]=[p,q(gT)]=Q([p,gT]),
		$$
		where the second equality uses the equivalence relation defining the associated bundle. In the local trivializations induced by the principal atlas, $Q$ has the form $(y,T)\mapsto(y,q(T))$, so it is a Lipschitz retraction onto $P\times_GG$.
		
		The inverse of $\iota:P\to P\times_G G$ is given by $\iota^{-1}([p,g])=pg$. In the local trivializations induced by the principal atlas, this map is represented by the identity and hence is Lipschitz. Therefore
		$$
		s=\iota^{-1}\circ Q\circ\sigma
		$$
		is a global Lipschitz section of $P\to Y$.
	\end{proof}
	
	For the map $\phi$ given by Lemma~\ref{prop:torsion-pullback}, the pullback bundle $\phi^*\cP^\circ$ is topologically trivial, and the pullback of the reduced Lipschitz atlas is again Lipschitz. Theorem~\ref{thm:lipschitz-section} therefore gives a global Lipschitz section of $\phi^*\cP^\circ$, hence a global Lipschitz family of exact section isometries.
	
	\begin{theorem}\label{cor:global-exact-family}
		If $n\geq3$ is odd, then there exist a smooth map $\phi:\Sph^n\to\Sph^n$ of degree $d\geq1$ and a Lipschitz map $z\mapsto A_z\in\Hom(E,\R^{n+1})$ such that
		\begin{equation}\label{eq:global-exact-family}
			A_z(E)=\phi(z)^\perp,
			\qquad
			A_z(S)=S_{\phi(z)}=K\cap\phi(z)^\perp
		\end{equation}
		for each $z\in\Sph^n$.
	\end{theorem}
	
	\begin{proof}
		Choose $d$ and $\phi$ as in Lemma~\ref{prop:torsion-pullback}. Let $\{(U_i,s_i^\circ)\}$ be the reduced Lipschitz atlas from Lemma~\ref{prop:identity-component}, and put $W_i=\phi^{-1}(U_i)$. The pullback atlas has local sections $z\mapsto(z,s_i^\circ(\phi(z)))$ and transition functions $t_{ji}^\circ\circ\phi$. These transition functions are Lipschitz since $t_{ji}^\circ$ is Lipschitz and $\phi$ is smooth on the compact sphere. By Lemma~\ref{prop:torsion-pullback}, the pullback bundle $\phi^*\cP^\circ$ is topologically trivial. Theorem~\ref{thm:lipschitz-section} therefore gives it a global Lipschitz section $\sigma$.
		
		Write $\sigma(z)=(z,A_z)$. On $W_i$, let $h_i:W_i\to G^\circ$ be the coordinate map of $\sigma$. Then
		$$A_z=s_i^\circ(\phi(z))h_i(z).$$
		Both factors on the right are bounded and Lipschitz, so $z\mapsto A_z$ is Lipschitz on every $W_i$. Since $\{W_i\}$ is a finite open cover of the compact sphere, a standard Lebesgue-number argument shows that $z\mapsto A_z$ is globally Lipschitz. Finally, $\sigma(z)\in(\phi^*\cP^\circ)_z$ gives $A_z\in\cP^\circ_{\phi(z)}\subset\cP_{\phi(z)}$. The definition of $\cP_{\phi(z)}$ now gives \eqref{eq:global-exact-family}.
	\end{proof}
	
	The preceding theorem is the only result of the gluing stage that will be used below.
	
	\vskip 25pt
	\section{\bf Proofs of the main results}
	\label{sec:degree-conclusion}
	\vskip 5pt
	
	We first prove Theorem~\ref{thm:odd-hyperplane}. Let $n\geq3$ be odd, and let $K\subset\R^{n+1}$ be an origin-symmetric convex body whose central hyperplane sections are pairwise linearly equivalent. Retain the notation of the preceding section: $S\subset E$ is the fixed model section, linearly equivalent to every section
	$S_u=K\cap u^\perp$. Choose
	$\phi$, $d$, and the Lipschitz family $z\mapsto A_z$ as in
	Theorem~\ref{cor:global-exact-family}.
	
	We use this family in a degree calculation. By \eqref{eq:section-gauge} and \eqref{eq:global-exact-family}, we have
	\begin{equation}\label{eq:Minkowski functional-exactness}
		p_K(A_z y)=p_{K\cap\phi(z)^\perp}(A_z y)=p_{A_z(S)}(A_z y)=p_S(y)
	\end{equation}
	for all $z\in\Sph^n$ and $y\in E$.

	For each $y\in E\setminus\{0\}$, \eqref{eq:Minkowski functional-exactness} gives $p_K(A_z y)=p_S(y)>0$, so the map
	$$
	w_y:\Sph^n\longrightarrow\Sph^n,
	\qquad
	w_y(z)=\frac{A_z y}{|A_z y|}
	$$
	is well defined. Since $A_z y\in\phi(z)^\perp$, one has
	\begin{equation}\label{eq:orthogonal-orbits}
		\ip{w_y(z)}{\phi(z)}=0.
	\end{equation}
	
	\begin{lemma}\label{prop:orbit-degree}
		For each $y\in E\setminus\{0\}$, $\degree w_y=d$.
	\end{lemma}

	\begin{proof}
		By \eqref{eq:orthogonal-orbits}, $\phi(z)$ and $w_y(z)$ are
		orthogonal unit vectors for every $z\in\Sph^n$. Hence
		$$
		H(z,t)=\cos\Bigl(\frac{\pi t}{2}\Bigr)\phi(z)+\sin\Bigl(\frac{\pi t}{2}\Bigr)w_y(z)
		$$
		defines a homotopy
		$H:\Sph^n\times[0,1]\to\Sph^n$ from $\phi$ to $w_y$.
		Therefore, by homotopy invariance,
		$\degree w_y=\degree\phi=d$.
	\end{proof}

	To pass from the degree of $w_y$ to Brouwer degree, let $A:\Sph^n\to\Hom(E,\R^{n+1})$ denote the map $z\mapsto A_z$, and extend it by positive homogeneity to
	$$
	\cA:\R^{n+1}\longrightarrow\Hom(E,\R^{n+1}),
	\qquad\cA(0)=0,\qquad\cA(tz)=tA_z
	$$
	for $z\in\Sph^n$ and $t>0$. For $y\in E$, define
	$$
	F_y:\R^{n+1}\longrightarrow\R^{n+1},\qquad F_y(x)=\cA(x)y.
	$$
	
	For $y\in E\setminus\{0\}$, \eqref{eq:Minkowski functional-exactness} gives $F_y(z)=A_z y\in p_S(y)\partial K$ for $z\in\Sph^n$.
	By Lemma~\ref{prop:orbit-degree}, $\degree w_y=d\geq 1$, and hence $w_y$ is surjective. Since $w_y(z)=F_y(z)/|F_y(z)|$ and $F_y$ is positively homogeneous, it follows that
	\begin{equation}\label{eq:cone-image}
		F_y(\Sph^n)=p_S(y)\partial K,
        \qquad
        F_y(\Ball^{n+1})=p_S(y)K.
	\end{equation}

	\begin{lemma}\label{lem:cone-lipschitz}
		The map $\cA$ is Lipschitz. Consequently, for each $y\in E$, $F_y$ is Lipschitz with $\Lip(F_y)\leq\Lip(\cA)|y|$.
	\end{lemma}
	
	\begin{proof}
		Put $\Lambda_0=\sup_{z\in\Sph^n}\|A_z\|$. Since the fixed Riemannian distance on $\Sph^n$ is bi-Lipschitz equivalent to the Euclidean chordal distance, there exists $\Lambda_1>0$ such that $\|A_z-A_w\|\leq\Lambda_1|z-w|$ for all $z,w\in\Sph^n$. Let $x,x'\in\R^{n+1}$. If $x=0$ or $x'=0$,
		then $\|\cA(x)-\cA(x')\|\leq\Lambda_0|x-x'|$. Otherwise, after
		interchanging $x$ and $x'$ if necessary, write
		$x=\alpha z$ and $x'=\beta w$, where $z,w\in\Sph^n$ and
		$0<\alpha\leq\beta$. Then we have $|\alpha-\beta|\leq|x-x'|$ and 
		$$
		\alpha|z-w|\leq|\alpha z-\beta w|+|\beta w-\alpha w|=|x-x'|+|\beta-\alpha|
		\leq2|x-x'|.
		$$
		Hence
		$$
		\begin{aligned}
			\|\cA(x)-\cA(x')\|
			&=\|\alpha A_z-\beta A_w\|\leq\alpha\|A_z-A_w\|
			+|\alpha-\beta|\|A_w\|\\
			&\leq\alpha\Lambda_1|z-w|
			+\Lambda_0|\alpha-\beta|
			\leq(2\Lambda_1+\Lambda_0)|x-x'|.
		\end{aligned}
		$$
		Thus $\cA$ is Lipschitz. Finally, for each $y\in E$,
		$$
		|F_y(x)-F_y(x')|=|(\cA(x)-\cA(x'))y|\leq\|\cA(x)-\cA(x')\||y|
		\leq\Lip(\cA)|y|\,|x-x'|.
		$$
		Hence $\Lip(F_y)\leq\Lip(\cA)|y|$.
	\end{proof}
	
	\begin{lemma}\label{prop:cone-degree}
    For every $y\in E\setminus\{0\}$,
    $$
        \degree(F_y,\Int\Ball^{n+1},v)=d,
        \qquad v\in\Int\bigl(p_S(y)K\bigr),
    $$
    while the degree is zero for $v\notin p_S(y)K$.
    \end{lemma}
	
	\begin{proof}
		Fix $y\in E\setminus\{0\}$ and write $r=p_S(y)$. Define $\widehat\Psi_r(0)=0$ and
		$$
		\widehat\Psi_r(t\theta)=tr\rho_K(\theta)\theta,
		\quad\text{for }t>0,\ \theta\in\Sph^n.
		$$
		This is an orientation-preserving homeomorphism of $\R^{n+1}$
		carrying $(\Ball^{n+1},\Sph^n)$ onto $(rK,r\partial K)$. Indeed,
		since $r\rho_K(\theta)>0$, the maps
		$$
		\Psi_s(0)=0,
		\qquad
		\Psi_s(t\theta)
		=t\bigl((1-s)+sr\rho_K(\theta)\bigr)\theta,
		\quad 0\leq s\leq1,
		$$
		form an isotopy through homeomorphisms from
		$\operatorname{Id}_{\R^{n+1}}$ to $\widehat\Psi_r$.
		
		Let $C_{w_y}:\Ball^{n+1}\to\Ball^{n+1}$ be the ordinary cone on $w_y$, defined by $C_{w_y}(tz)=tw_y(z)$. For $z\in\Sph^n$, \eqref{eq:cone-image} gives
		$F_y(z)\in r\partial K$. This, together with that $w_y(z)=F_y(z)/|F_y(z)|$ and the definition of $\rho_K$, yields that
		$$
		F_y(z)
		=|F_y(z)|w_y(z)=r\rho_K(w_y(z))w_y(z)
		=\widehat\Psi_r(w_y(z)).
		$$
		By positive homogeneity, $F_y|_{\Ball^{n+1}}=\widehat\Psi_r\circ C_{w_y}$.
        
		Let $v\in\Int(rK)$ and set
		$v_0=\widehat\Psi_r^{-1}(v)\in\Int\Ball^{n+1}$. By invariance of Brouwer degree under the orientation-preserving homeomorphism $\widehat\Psi_r$ and Lemmas~\ref{lem:radial-cone-degree} and~\ref{prop:orbit-degree},
		$$
		\degree(F_y,\Int\Ball^{n+1},v)=\degree(C_{w_y},\Int\Ball^{n+1},v_0)=\degree w_y=d.
		$$
		If $v\notin rK$, then $v\notin F_y(\Ball^{n+1})$ by \eqref{eq:cone-image}, so the degree is zero.
	\end{proof}
	
	We next use Lemma~\ref{prop:cone-degree} and the signed degree formula to show that certain powers of $p_S$ are homogeneous polynomials.
	
	\begin{lemma}\label{prop:polynomial-powers}
		For each integer $k\geq0$, the function $y\mapsto p_S(y)^{n+1+2k}$ is a homogeneous polynomial on $E$ of degree $n+1+2k$.
	\end{lemma}
	
	\begin{proof}
		We divide the proof into two steps. Fix an integer $k\geq0$.
		
		\textbf{Step 1.} Derive a scaling identity for $p_S(y)$ from the
	degree of $F_y$.
	
	Fix $y\in E\setminus\{0\}$, set $r=p_S(y)$, and let
	$\Omega=\Int\Ball^{n+1}$. By Lemma~\ref{lem:cone-lipschitz},
	$F_y$ is Lipschitz on $\R^{n+1}$. By \eqref{eq:cone-image},
	$F_y(\overline\Omega)=rK$ and
	$F_y(\partial\Omega)=r\partial K$. Since
	$\vol_{n+1}(\partial\Omega)=0$, we can apply
	Lemma~\ref{prop:signed-degree-formula} to $F_y$ on $\Omega$ with the
	bounded Borel function equal to $|v|^{2k}$ on $rK$ and zero
	outside, where $|v|^0=1$. Using Lemma~\ref{prop:signed-degree-formula}, Lemma~\ref{prop:cone-degree} and then
	making the change of variables $v=ru$, we obtain
	\begin{align}\label{eq:moment-degree-identity}
		\int_\Omega |F_y(x)|^{2k}\det DF_y(x)\dd x
		&=\int_{\R^{n+1}}|v|^{2k}
		\degree(F_y,\Omega,v)\dd v=d\int_{rK}|v|^{2k}\dd v\\
		&=dr^{n+1+2k}\int_K|u|^{2k}\dd u=dp_S(y)^{n+1+2k}\int_K|u|^{2k}\dd u.
	\notag
	\end{align}
	
	\textbf{Step 2.} Show that the left-hand side of \eqref{eq:moment-degree-identity} is a homogeneous polynomial in $y$.
	
	By Rademacher's theorem~\cite[Theorem~3.2]{EvansGariepy2015}, there exists a null set $Z\subset\Omega$ such that $\cA$ is differentiable at every $x\in\Omega\setminus Z$. This set is independent of $y$. From $F_y(x)=\cA(x)y$ we obtain
	$$
	DF_y(x)\xi=\bigl(D\cA(x)\xi\bigr)y,
	\qquad x\in\Omega\setminus Z,\quad \xi\in\R^{n+1}.
	$$
	Thus, for each $x\in\Omega\setminus Z$, the map
	$y\mapsto DF_y(x)$ is linear. Hence
	$y\mapsto\det DF_y(x)$ and
	$y\mapsto|F_y(x)|^{2k}
	=\langle\cA(x)y,\cA(x)y\rangle^k$
	are homogeneous polynomials of degrees $n+1$ and $2k$, respectively.
	Their product is therefore a homogeneous polynomial of degree
	$n+1+2k$.
	
	Since $\cA$ is bounded on $\Ball^{n+1}$ and $D\cA$ is essentially bounded on $\Omega$, the coefficients of this polynomial are integrable over $\Omega$. Hence
	$$
	P_k(y):=\int_\Omega |F_y(x)|^{2k}\det DF_y(x)\dd x
	$$
	is a homogeneous polynomial of degree $n+1+2k$. By \eqref{eq:moment-degree-identity}, for $y\neq0$,
	$$
	P_k(y)=c_kp_S(y)^{n+1+2k},
	\qquad
	c_k:=d\int_K|u|^{2k}\dd u>0.
	$$
	Therefore $p_S^{n+1+2k}=P_k/c_k$ on $E\setminus\{0\}$. Both sides
	vanish at the origin, so the identity holds on all of $E$. This proves the lemma.
	\end{proof}
	
	To deduce from the cases $k=0$ and $k=1$ of Lemma~\ref{prop:polynomial-powers} that $p_S^2$ is a quadratic form, we use the following elementary consequence of unique factorization.
	
	\begin{lemma}\label{lem:consecutive-powers}
    Let $a\geq1$ be an integer, and let $P,Q$ be nonzero elements of
    a unique factorization domain such that $Q^a=P^{a+1}$. Then
    $P$ divides $Q$.
\end{lemma}
	
	\begin{proof}
		For each irreducible element $\pi$, let $\nu_\pi(P)$ and $\nu_\pi(Q)$ denote its multiplicities in $P$ and $Q$. Comparing multiplicities in $Q^a=P^{a+1}$ gives $a\nu_\pi(Q)=(a+1)\nu_\pi(P)$, and hence
		$\nu_\pi(Q)\geq\nu_\pi(P)$. Thus every irreducible factor of $P$ occurs in $Q$ with at least the same multiplicity, so $P\mid Q$.
	\end{proof}
	
	\begin{lemma}\label{prop:model-euclidean}
		The function $y\mapsto p_S(y)^2$ is a positive-definite quadratic form on $E$. Consequently, $S$ is an ellipsoid.
	\end{lemma}
	
	\begin{proof}
		Let $a=(n+1)/2$. Since $n\geq3$ is odd, $a$ is a positive integer. Lemma~\ref{prop:polynomial-powers}, applied with $k=0$ and $k=1$, shows that
		$$
		P(y):=p_S(y)^{n+1}=p_S(y)^{2a},
		\qquad
		Q(y):=p_S(y)^{n+3}=p_S(y)^{2(a+1)}
		$$
		are nonzero homogeneous polynomials. Pointwise on $E$, we have $Q^a=P^{a+1}$.
		Since $\R$ is infinite, this pointwise equality is an identity in the ring $\R[E]$ of polynomial functions $E\to\R$. After choosing linear coordinates $y_1,\ldots,y_n$ on $E$, this ring is isomorphic to $\R[y_1,\ldots,y_n]$ and hence is a unique factorization domain~\cite[Section~9.3, Theorem~7]{DummitFoote2004}. Lemma~\ref{lem:consecutive-powers} therefore gives $Q=PR$ for some polynomial $R$.
		
		Since $P$ and $Q$ are homogeneous of degrees $n+1$ and $n+3$, respectively, comparison of homogeneous components in $Q=PR$ shows that $R$ is homogeneous of degree two. For $y\neq0$, we have $P(y)>0$ and hence $$R(y)=\frac{Q(y)}{P(y)}=p_S(y)^2.$$ Both sides vanish at the origin, so this identity holds on all of $E$. Thus $p_S^2=R$ is a quadratic form. Since $p_S(y)>0$ for $y\neq0$, this quadratic form is positive definite. Consequently, $$S=\{y\in E:R(y)\leq1\}$$ is an ellipsoid.
	\end{proof}

	\begin{proof}[Proof of Theorem~\ref{thm:odd-hyperplane}]
    By Lemma~\ref{prop:model-euclidean}, $S$ is an ellipsoid. Since every central hyperplane section of $K$ is linearly equivalent to
    $S$, every such section is an ellipsoid.

    Let $x,y\in\R^{n+1}$. Since $n\geq3$, there is a hyperplane $H\subset\R^{n+1}$ containing $\Span\{x,y\}$. The unit ball of the restricted norm $p_K|_H$ is $K\cap H$, which is an ellipsoid. Hence $p_K|_H$ is induced by an inner product and
    $$
        p_K(x+y)^2+p_K(x-y)^2
        =2p_K(x)^2+2p_K(y)^2.
    $$
    Since $x,y$ are arbitrary, $p_K$ satisfies the parallelogram identity. The Jordan--von Neumann
    theorem~\cite{JordanVonNeumann1935} therefore implies that $p_K$ is induced by an inner product. Thus its unit ball $K$ is an ellipsoid.
    \end{proof}
	
	\begin{proof}[Proof of Theorem~\ref{thm:main}]
		If $n$ is even, the conclusion follows from Gromov's result \cite[Theorem~1]{Gromov1967}.
		
		It remains to consider the case when $n$ is odd. By the codimension-one reduction explained in the introduction, it suffices to assume that $\dim X=n+1$. Let $B_X$ be the unit ball of $X$. Its central
		hyperplane sections are the unit balls of the $n$-dimensional
		subspaces of $X$ and hence are linearly equivalent. Identifying
		$X$ linearly with $\R^{n+1}$,
		Theorem~\ref{thm:odd-hyperplane} shows that $B_X$ is an ellipsoid.
		Thus $X$ is a Hilbert space.
	\end{proof}
	
	\begin{proof}[Proof of Theorem~\ref{cor:convex-main}]
    Equip $\R^N$ with the norm $p_K$. The unit ball of the restricted norm on any $n$-dimensional subspace of $\R^N$ is its intersection with $K$. Hence the hypothesis says precisely that all $n$-dimensional subspaces of $(\R^N,p_K)$ are linearly isometric. Theorem~\ref{thm:main} therefore implies that $p_K$ is induced by an inner product, so its unit ball $K$ is an ellipsoid.
    \end{proof}

	\begin{remark}
		The oddness of $n$ is used at two distinct points.
		
		First, $n-1$ is positive and even, so
		$\pi_{n-1}(G^\circ)$ is finite and the bundle class can be killed
		by pullback along a suitable positive-degree self-map of
		$\Sph^n$.
		
		Second, the cases $k=0$ and $k=1$ of
		Lemma~\ref{prop:polynomial-powers} show that $p_S^{n+1}$ and
		$p_S^{n+3}$ are polynomials. Since $n$ is odd, these are two
		consecutive powers of $p_S^2$, allowing the unique-factorization
		argument to show that $p_S^2$ is a polynomial.
	\end{remark}
	
	\vskip 20pt
	\vskip3pt \noindent{\bf Declaration on the use of AI}: The authors had reduced the main problem to proving Theorem~\ref{cor:global-exact-family} before using generative AI tools. An approach to that theorem subsequently emerged through extensive interactions with ChatGPT 5.5 Pro and ChatGPT 5.6 Pro. 
	The initial draft of Section 3 and the corresponding parts in Section 2 were generated by GPT 5.6 Sol following this approach, and subsequently checked and rewritten by the authors. 
	GPT-5.6 Sol was also used to improve the exposition. The authors take full responsibility for the mathematical content and the final text.
	
	\vskip3pt \noindent{\bf Conflict of Interest.} The authors declare no conflict of interest.
	\vskip3pt \noindent{\bf Data Availability.} Not applicable.

\end{document}